\documentclass[11pt]{amsart}
\usepackage{amsmath,amsthm,amssymb,amscd}
\usepackage{graphicx}
\usepackage[all]{xy}
\usepackage[margin=1in]{geometry}
\usepackage{bm}
\usepackage[usenames, dvipsnames]{color}

\usepackage[T1]{fontenc}

\theoremstyle{plain}
\newtheorem{thm}{Theorem}
\newtheorem{prop}[thm]{Proposition}
\newtheorem{cor}[thm]{Corollary}
\newtheorem{lem}[thm]{Lemma}

\theoremstyle{definition}
\newtheorem*{defn}{Definition}
\newtheorem*{claim}{Claim}
\newtheorem*{exmp}{Example}

\theoremstyle{remark}
\newtheorem*{rem}{Remark}

\newtheorem*{ack}{Acknowledgments}

\numberwithin{equation}{section}

\newcommand{\RP} {\R\Pj}                                 
\newcommand{\R} {{\mathbb R}}                              
\newcommand{\Pj} {{\mathbb P}}                             

\newcommand{\pos} {\mathcal{P}}		

\newcommand{\PeG} {\pos_{\eG}}		

\newcommand{\K} {\mathcal{K}}   
\newcommand{\KG} {\K G}
        
\newcommand{\rKG} {{\textcolor{BrickRed}{\mathbb K}G}}
\newcommand{\rK} {{\textcolor{BrickRed}{\mathbb K}}}

\newcommand{\rKGn} [1] {{\textcolor{BrickRed}{\mathbb K}}[{#1}]}
\newcommand{\eG} {EG}

\newcommand{\rc} [2] {{  {\mathfrak {S} }}_{#2}(#1)}   
\newcommand{\rKGi} {{\textcolor{BrickRed}{\mathbb K^*}G}}
\newcommand{\PeGi} {\pos^*_{EG}}

\begin{document}


\title{Colorful Graph Associahedra}

\author{Satyan L.\ Devadoss}
\address{S.\ Devadoss: University of San Diego, San Diego, CA 92110}
\email{devadoss@sandiego.edu}

\author{Mia Smith}
\address{Mia Smith: Proof School, San Francisco, CA 94103}
\email{msmith@proofschool.org}

\begin{abstract}
Given a graph $G$, the graph associahedron is a simple convex polytope whose face poset is based on the connected subgraphs of $G$. With the additional assignment of a color palette, we define the colorful graph associahedron, show it to be a collection of simple abstract polytopes, and explore its properties. 
\end{abstract}



\maketitle

\baselineskip=17pt

%
%
\section{Introduction}  \label{s:intro}

Given a finite graph $G$, the graph associahedron $\KG$ is a simple polytope \cite{cd} whose face poset is based on tubes, the connected subgraphs of $G$.  For special examples of graphs, $\KG$ becomes well-known, sometimes classical:  when $G$ is a path, a cycle, or a complete graph, $\KG$ results in the associahedron, cyclohedron, and permutohedron, respectively.  Figure~\ref{f:kwtubes} shows the examples for a path and a cycle with three nodes.
\begin{figure}[h]
\includegraphics[width=.9\textwidth]{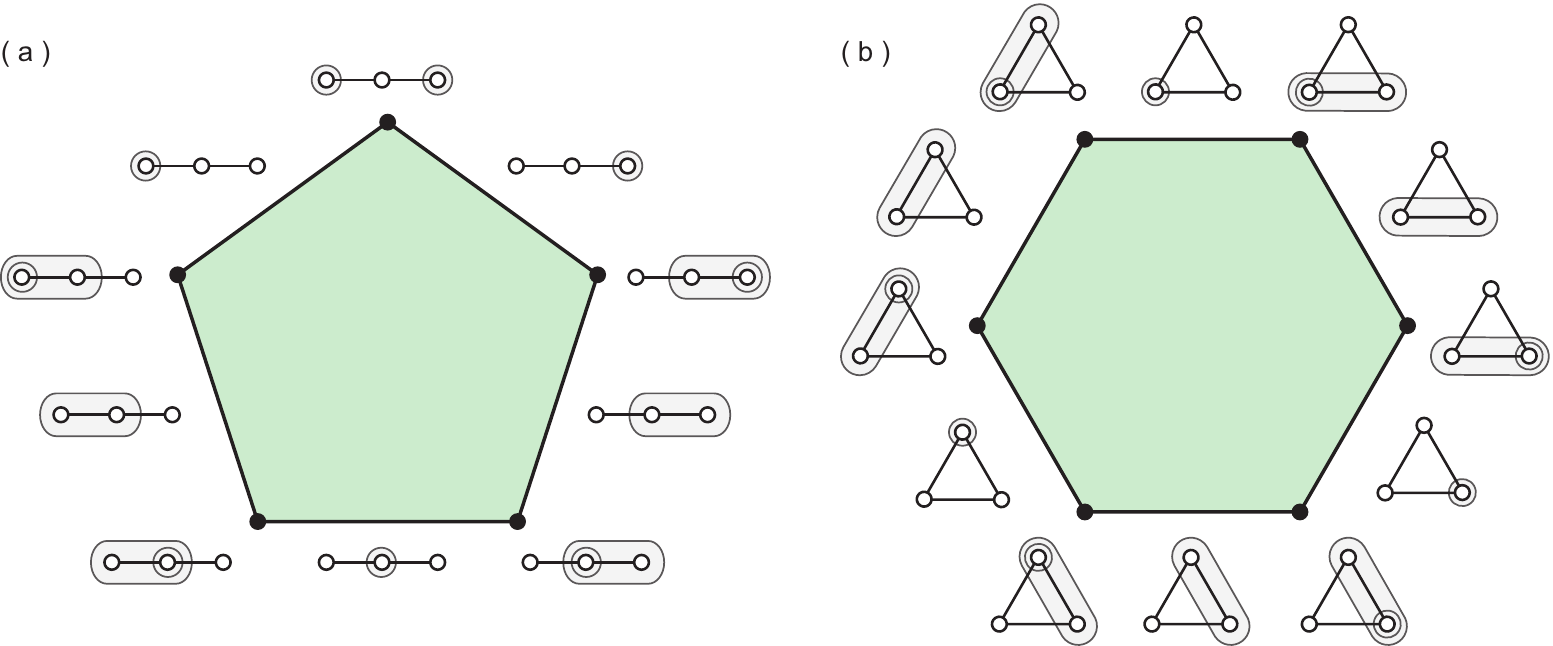}
\caption{Graph associahedron of the (a) path and (b) cycle with three nodes.}
\label{f:kwtubes}
\end{figure}
These polytopes were first motivated by De Concini and Procesi in their work on `wonderful' compactifications of hyperplane arrangements \cite{dp}.  They make appearances in numerous areas, including geometric group theory \cite{djs}, real moduli space of curves \cite{cd}, Heegaard Floer homology \cite{blo}, and biological statistics \cite{mps}.  It is not surprising to see $\KG$ in such a broad range of subjects since the structures of these polytopes capture and expose the fundamental concepts of connectivity and nestings.  There have been numerous extensions, including nestohedra \cite{fei} and Postnikov's generalized permutohedra \cite{pos}.  

This work deals with another such generalization, now with the additional feature of coloring the tubes.  These \emph{colorful} graph associahedra $\rKG$ are  motivated by the work of Araujo-Pardo, Hubard, Oliveros, and Schulte \cite{ahos2} that explores coloring diagonals of triangulations\footnote{Recent work by  Lubiw, Mas\'{a}rov\'{a}, and Wagner \cite{orbit} explores colored triangulations of an arbitrary planar point set, establishing connectivity  through sequences of color diagonal flips, generalizing the work of Lawson \cite{law}.} of an $n$-gon (associahedron) and centrally symmetric triangulations of a $2n$-gon (cyclohedron).  These generalize to posets that form simple abstract polytopes, whose automorphism groups are studied in detail.
The natural bijections \cite{cd} between such diagonalizations  (of polygons and centrally symmetric ones) to tubes (on paths and cycles) stimulate our work.

Figure \ref{f:ksg-path} shows an example of $\rKG$ for the path of 3 nodes, resulting in a 10-gon; contrast this with Figure~\ref{f:kwtubes}(a).
Categorical complexity increases with the number of nodes: For the path with four nodes, $\rKG$ is no longer a polytope but a genus-four handlebody.  And for paths with five or more nodes, $\rKG$ is not even a manifold but an abstract polytope.  

\begin{figure}[h]
\includegraphics{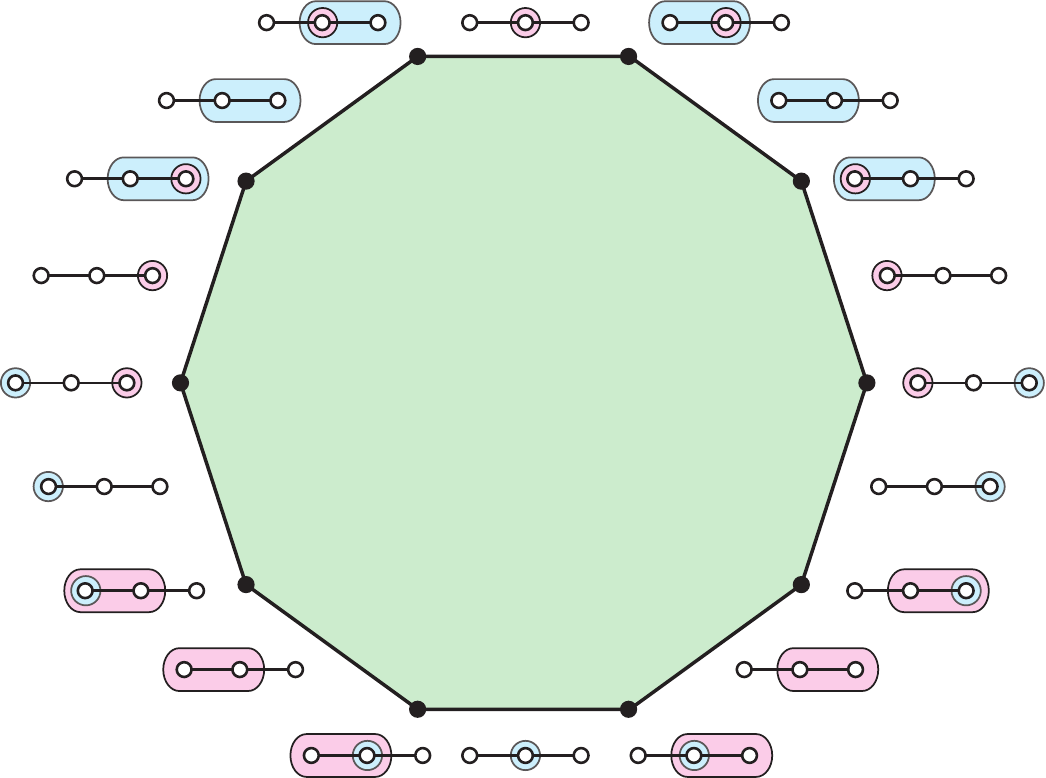}
\caption{Colorful graph associahedron of the path with three nodes.}
\label{f:ksg-path}
\end{figure}

This paper extends the results of  \cite{ahos2} to arbitrary simple graphs $G$.  We first show how the connectivity levels of $G$ stand in obstruction to simply coloring the tubes of $G$.  Thus, colored tubes require larger machinery to be well-defined, notably a color template and associated compatibility structures.  With this in hand, for a simple graph $G$ with $n$ nodes, a definition of the colorful graph associahedron  $\rKG$  is offered and the strongest  possible result proven:  $\rKG$  is a simple abstract polytope of rank $n-1$.

An overview of this paper is as follows: Section~\ref{s:basics} provides  background, while exploring issues of connectivity and cores.  Section~\ref{s:colors} introduces colors, along with the structures of palettes and templates necessary for a well-defined formulation of the colorful graph associahedron.  Section~\ref{s:colorful} showcases the main theorem along with examples and properties, whereas Section~\ref{s:disconnect} deals with disconnected graphs and issues of regularity.  The proofs of the main results are  provided in Section~\ref{s:proofs}.

\begin{ack}
We would like to thank Colin Adams and Egon Schulte for helpful conversations, and support from Williams College and Harvey Mudd College.  The research for this work was finished in summer 2016, and the bulk of the writing occurred in fall 2017 at MSRI, during the program on geometric and topological combinatorics, to which we are grateful.  Devadoss was partially supported by the John Templeton Foundation grant 51894 and takes full responsibility for the tardiness in bringing this work to the mathematics community.
\end{ack}

%
%
\section{Graphs and Definitions}  \label{s:basics}
\subsection{Tubes}

Throughout the paper, let $G$ be a simple graph, following the lead in \cite{cd}.

\begin{defn}
A \emph{tube} of $G$ is a set of nodes whose induced graph is a proper, connected subgraph.
Two tubes are \emph{compatible} if one properly contains the other, or if they are disjoint and cannot be connected by an edge of $G$. A $k$-\emph{tubing} of $G$ is a set of $k$ pairwise compatible tubes.
\end{defn}

\begin{thm} \cite{cd}
Given a simple graph $G$ with $n$ nodes, the \emph{graph associahedron} $\KG$ is a simple convex polytope of dimension $n-1$ whose face poset is isomorphic to the set of valid tubings of $G$, ordered such that $T \prec T'$ if $T$ is obtained from $T'$ by adding tubes. 
\label{t:classic-ga}
\end{thm}

The codimension $k$-faces correspond to $k$-tubings of $G$; in particular, the vertices of $\KG$ correspond to \emph{maximal} tubings of $G$, those containing $n-1$ tubes. For a maximal tubing, there is exactly one node of $G$ not contained in any tube, called the \emph{universal node}.  For technical reasons discussed later, we define the entire graph $G$ itself as the \emph{universal tube}.\footnote{This is not an official tube since it is not a proper subgraph of $G$.}  
It is tempting to define $\rKG$ analogous to $\KG$, as the poset of valid color tubings of $G$, ordered by adding and removing color tubes. As we discuss below, this method fails due to issues of connectivity.  

\begin{defn}
A graph $G$ has \emph{connectivity} $k$ if there exists a set of $k$ nodes whose removal disconnects the graph, but there is no set of $k-1$ nodes whose removal disconnects it. 
\end{defn}
 
A hint of the larger problem at hand appears even when considering cycles (connectivity 2) rather than just paths (connectivity 1). Figure~\ref{f:ksg-cycle} displays $\rKG$ for a cycle with three nodes.
Unlike Figure~\ref{f:ksg-path}, the resulting object yields two isomorphic copies of a hexagon, each with a different color schemata.  In particular, when $G$ is an $n$ cycle,  $\rKG$ is made of $n-1$ isomorphic copies.  This small disparity between paths and cycles hints at a looming failure: when the connectivity of $G$ increases beyond 2, the ability to define $\rKG$ simply as the poset of valid color tubings of $G$  becomes impossible.

\begin{figure}[h]
\includegraphics[width=.9\textwidth]{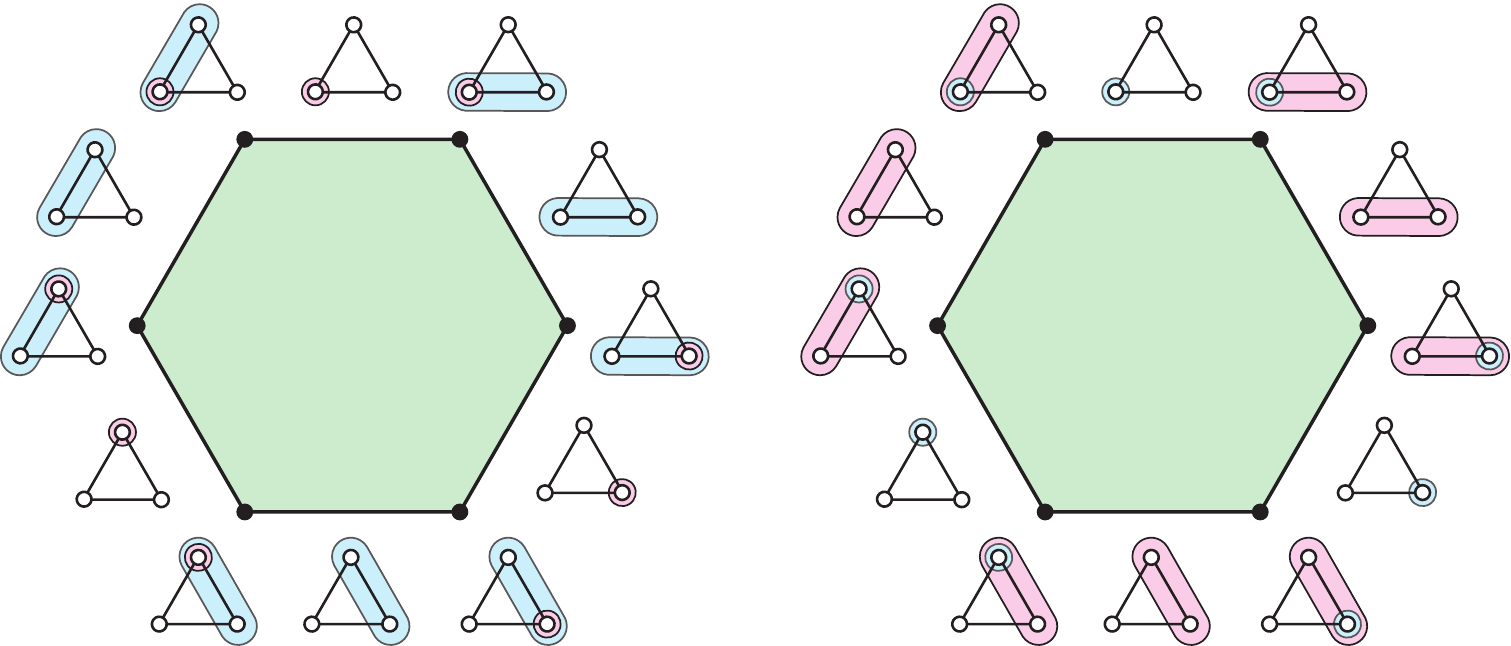}
\caption{Colorful graph associahedron of the cycle with three nodes.}
\label{f:ksg-cycle}
\end{figure}

\subsection{Inner and Outer}

The rest of this section is devoted to examining this phenomena and establishing the definitions necessary for the proper formulation of colorful graph associahedra.

\begin{defn}
Given a graph $G$ with $n$ nodes and connectivity $k$, a tube of $G$ with $n-k$ nodes or less is an \emph{inner} tube, and a tube with more than $n-k$ nodes is an \emph{outer} tube.
\end{defn}

\begin{prop} \label{l:innerinsiderouter}
For any inner tube $t_i$ and outer tube $t_u$ of tubing $T$, we have $t_i\subset t_u$, justifying the nomenclature of this terminology.
\end{prop}

\begin{proof}
It is either the case that $t_i$ and $t_u$ are disjoint or that one is a subset of the other. The total number of nodes in their union is greater than $n-k$, which means that the subgraph induced by the union of their nodes is connected. Hence, they cannot be disjoint and so $t_i\subset t_u$. 
\end{proof}

\begin{prop}
\label{p:maxtubing}
For a graph $G$ with $n$ nodes and connectivity $k$, every maximal tubing contains $k-1$ outer tubes and $n-k$ inner tubes. In particular, for all $n-k < i < n$, there is precisely one outer tube containing $i$ nodes.
\end{prop}

\begin{proof}
By the  logic in Proposition~\ref{l:innerinsiderouter}, given two outer tubes, one must be contained in the other. This means no two outer tubes contain the same number of nodes, so there are at most $k-1$ outer tubes.
Suppose these were fewer; then, for some number of nodes $i$ with $n-k < i < n$, there would be no tube containing precisely $j$ nodes in $T$. Consider the smallest tube $t$ with greater than $j$ nodes; this may be the universal tube. There exist at least two nodes in $t$ that are not in any $t'\subset t$. Taking $t$ and removing one of these two nodes yields a tube which is compatible with all other tubes of $T$, contradicting the maximality of the tubing. 
\end{proof}

\begin{exmp}
Figure~\ref{f:outer} shows three examples of graphs with 4 nodes, with connectivity 3, 2, and 1, respectively.  The left side of each pair shows a maximal tubing, whereas the right side depicts the outer tubes of that tubing.  Note that outer tubes are based on the underlying \emph{connectivity} of the graph, rather than properties such as nestings of tubes.
\end{exmp}

\begin{figure}[h]
\includegraphics[width=.9\textwidth]{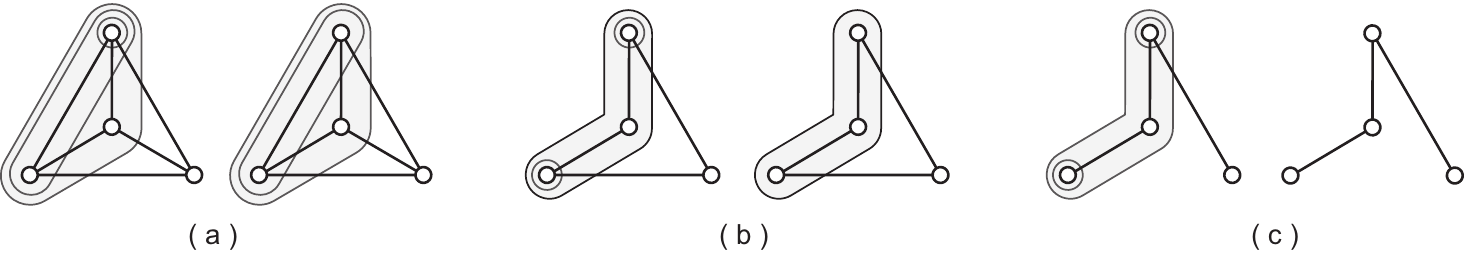}
\caption{Three pairs of graphs with four nodes each; the left side displays a maximal tubing, the right side the outer tubes of this tubing.}
\label{f:outer}
\end{figure} 

\subsection{Cores}

With connectivity being the first ingredient, the second concept  needed for constructing $\rKG$ is a manipulation algorithm for tubes and graphs.  The following definition is motivated by the notion of \emph{reconnected complement} as defined in \cite{cd}; indeed, this construct can be reformulated as an iterative reconnected complement operation.  

\begin{defn} 
Given a graph $G$, a tubing $T$, and a tube $t$ of $T$, let $\rc t T$ be the \emph{core graph} of $t$ in $T$:  The nodes of  $\rc t T$ are those of $t$ not contained in any other tubes $t' \subset t$ of $T$.  There is an edge between such two nodes if they are adjacent in $G$ or if they are connected via a path in some $t' \subset t$.
\end{defn}

\begin{rem}
As an exceptional case, extend this to the core graph of the universal tube $G$, denoted by $\rc G T$, whose nodes are those not contained in any tube of $T$.
\end{rem}

\begin{exmp}
Figure~\ref{f:recon}(a) shows an example of a tubing $T$ of $G$ with two tubes, the larger one denoted as $t$.  Part (b) shows the core  $\rc t T$ and  (c) the core  $\rc G T$ of the universal tube. 
\end{exmp}

\begin{figure}[h]
\includegraphics[width=.9\textwidth]{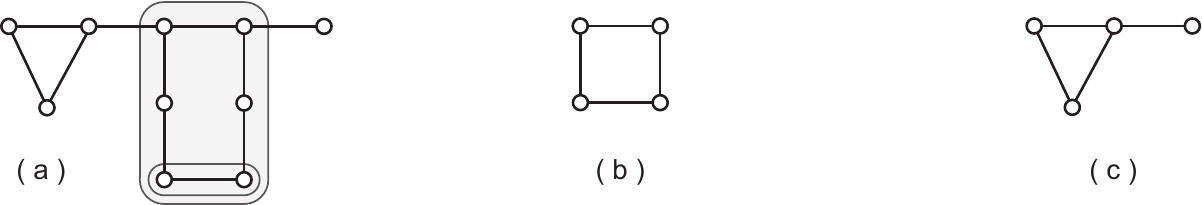}
\caption{A tubing and two core graphs.}
\label{f:recon}
\end{figure}

%
%
\section{Graphs and Colors}  
\label{s:colors}
\subsection{Palettes and Templates}

We now introduce color to the discussion. 
Araujo-Pardo, Hubard, Oliveros, and Schulte \cite{ahos2} construct colorful associahedra and cyclohedra using a set of distinct colors for the labeling of tubes.  We extend this to arbitrary color collections.

\begin{defn}
For a graph $G$ with $n$ nodes, consider its \emph{color palette}, a collection of $n-1$ colors where repeats are permitted. If the  palette has $n-1$ \emph{distinct} colors, we say it is a \emph{full} color palette.
 A \emph{color tube} is a tube of $G$ with a choice of color from the palette.  The colors of the tubes in a \emph{color $k$-tubing} are a choice of $k$ colors from the palette.
\end{defn}

The core graphs form the basis for how color is distributed. In particular, the following bijection allows us to recast tubings of $G$ entirely in the realm of core graphs.

\begin{prop} \label{p:core}
Given a graph $G$, a tubing $T$, and a tube $t$ of $T$, the tubes of  $\rc t T$ are in bijection with  tubes compatible with $T$ and strictly contained in $t$, but not contained in any other tube $t' \subset t$.
\end{prop}

\begin{proof}
By the construction of $\rc t T$, there is a bijection $\phi$ between the nodes of the core and the nodes of $G$ contained in $t$ but not in any $t' \subset t$.  For a tube $h$ compatible with $T$, its nodes can be partitioned into two sets: set $h_0$ consisting of all nodes contained in some $t'\subset t$, and set $h_1$ consisting of the rest. Mapping $h$ to $\phi(h_1)$ gives our desired bijection of tubes.
\end{proof}

This framing now affords the opportunity to combine the notion of cores and connectivity to define coloring tubes in a proper manner.

\begin{defn}
Let $G$ be a graph with a color palette. A $\emph{color template}$ is the pair $(T,\mathcal O)$: $T$ is an uncolored tubing of $G$, and $\mathcal O$ is a partition of the color palette amongst the tubes of $T$, where the colors associated to each tube $t$ satisfying the following:
\begin{enumerate}
\item The first block is the color of $t$.
\item Let $k_t$ be the connectivity of the core  $\rc t T$. The next $k_t-1$ blocks are the $\emph{outer blocks}$, and contain $k_t-1$  colors to be used on the outer tubes of $\rc t T$.
\item The last block, called the $\emph{inner block}$, contains $n-k_t$ colors to be used on the inner tubes of the core $\rc t T$. 
\end{enumerate}
\begin{figure}[h]
\includegraphics{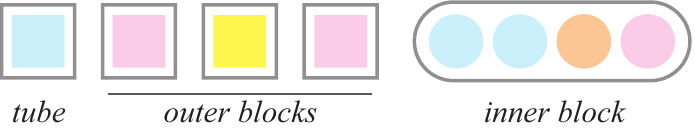}
\end{figure}
The template fixes the color orders for outer blocks whereas the colors in the inner block are unordered.  In the case of the universal tube, the color template identification is similar except for one change: the universal tube is uncolored. 
\end{defn}

\begin{figure}[b]
\includegraphics[width=\textwidth]{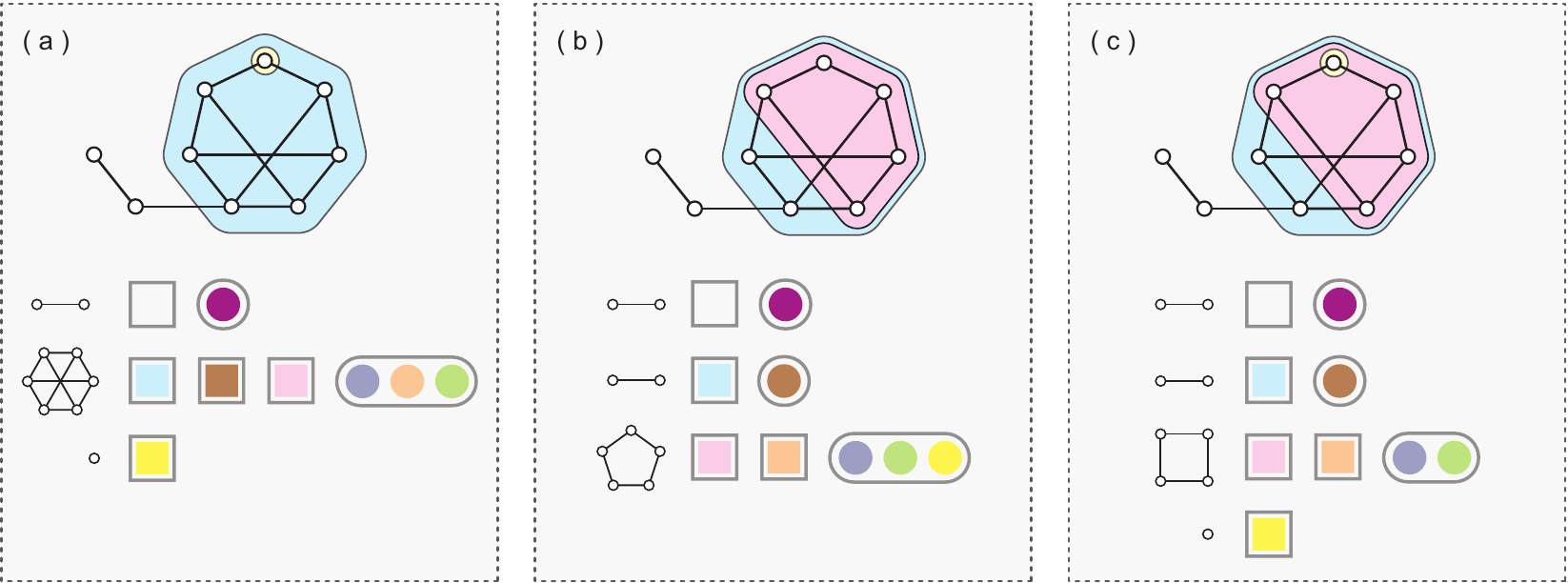}
\caption{Examples of color templates for three tubings.}
\label{f:template-e2}
\end{figure}

\begin{exmp}
Figure~\ref{f:template-e2} shows three examples of color templates $(T,\mathcal O)$ for three distinct tubings of a graph. The color palette uses eight distinct colors and each template partitions the color palette as expected. Each row displays  the core graph $\rc t T$ and the color blocks associated to each tube $t$, with the top row corresponding to the universal tube.
The first block of each row is the color of the tube, the only `visible' color in the shading of the tube.
\end{exmp}

\subsection{Compatibility}

These color templates will form the elements of our poset for $\rKG$.  A template can be thought of as a organizer: as our tubing $T$ is altered, by adding and removing tubes, the template gives us rules for success to ensure consistency. 
For color templates $(T, \mathcal O)$ and $(T', \mathcal O')$ to be compatible, we first require tube compatibility in the classical sense, where tubing $T$ can be formed from $T'$ by adding tubes.  With the addition of color,  the partitions of the color palette $\mathcal O$ and $\mathcal O'$ also need to be compatible, where  $\mathcal O$ can be formed by further partitioning $\mathcal O'$ in a way that respects the color order. This is now made precise.

\begin{rem}
For the color template  compatibility of $(T, \mathcal O)$ and $(T', \mathcal O')$, it is sufficient to consider when $T$ differs from $T'$ by only one tube.   For all other cases, an iterative procedure can be used to extend the following definition.
\end{rem}

\begin{defn}
Let $t$ be a tube of graph $G$, with tubings $T'$ and $T = T' \cup \{t\}$.  Moreover, let $t_*$ be the smallest tube in $T$ that contains $t$; if no such tube exists, let $t_*$ be the universal tube $G$. Two color template  $(T, \mathcal O) \prec (T', \mathcal O')$ are \emph{compatible} if the following holds:

\begin{enumerate}
\item In $\mathcal O$, the color blocks for all tubes in $T \setminus \{t, t_*\}$ must be the same as in $\mathcal O'$.

\item The $k=|\rc{t_*}{T}|$ colors of $t_*$ in $\mathcal O$ are the first $k$ colors of $t_*$ in $\mathcal O'$, preserving  partial order. If inner block colors are included in the first $k$ colors, then any of these colors may be selected.

\item The $|\rc{t}{T}|$ colors associated to $t$ are colors $k+1$ to $|\rc{t_*}{T'}|$ of $t_*$ in $\mathcal O'$, preserving  partial order. If the first block of $t$ together with the outer blocks in $\rc{t}{T}$ outnumber the outer blocks amongst colors $k+1$ to $|\rc{t_*}{T'}|$ of $t_*$ in $\mathcal O'$, then any inner block colors may be selected.
\end{enumerate}
\end{defn}

\begin{exmp}
Figure~\ref{f:compatibility} shows that the color template in (a) is compatible with the color template in (e) via the iterative addition of compatible tubes. Observe that at each step, when a tube $t$ is added, the colors for $t$ are selected starting from the end of the ordered partition associated to $t_*$. Progressing from (b) to (c),  $\rc t T$ has connectivity 2 but the last four colors associated to $t_*$ in $T'$ all lie in the inner block; per requirement (3), two colors are arbitrarily selected from the inner block, blue for the tube $t$ and pink for single outer block. Progressing from (c) to (d),  $\rc{t_*}{T}$ has higher connectivity than $\rc{t_*}{T'}$ and thus more outer blocks. Per requirement (2), green is arbitrarily selected to fill the outer block $t_*$.  
\end{exmp}

\begin{figure}[h]
\includegraphics[width=\textwidth]{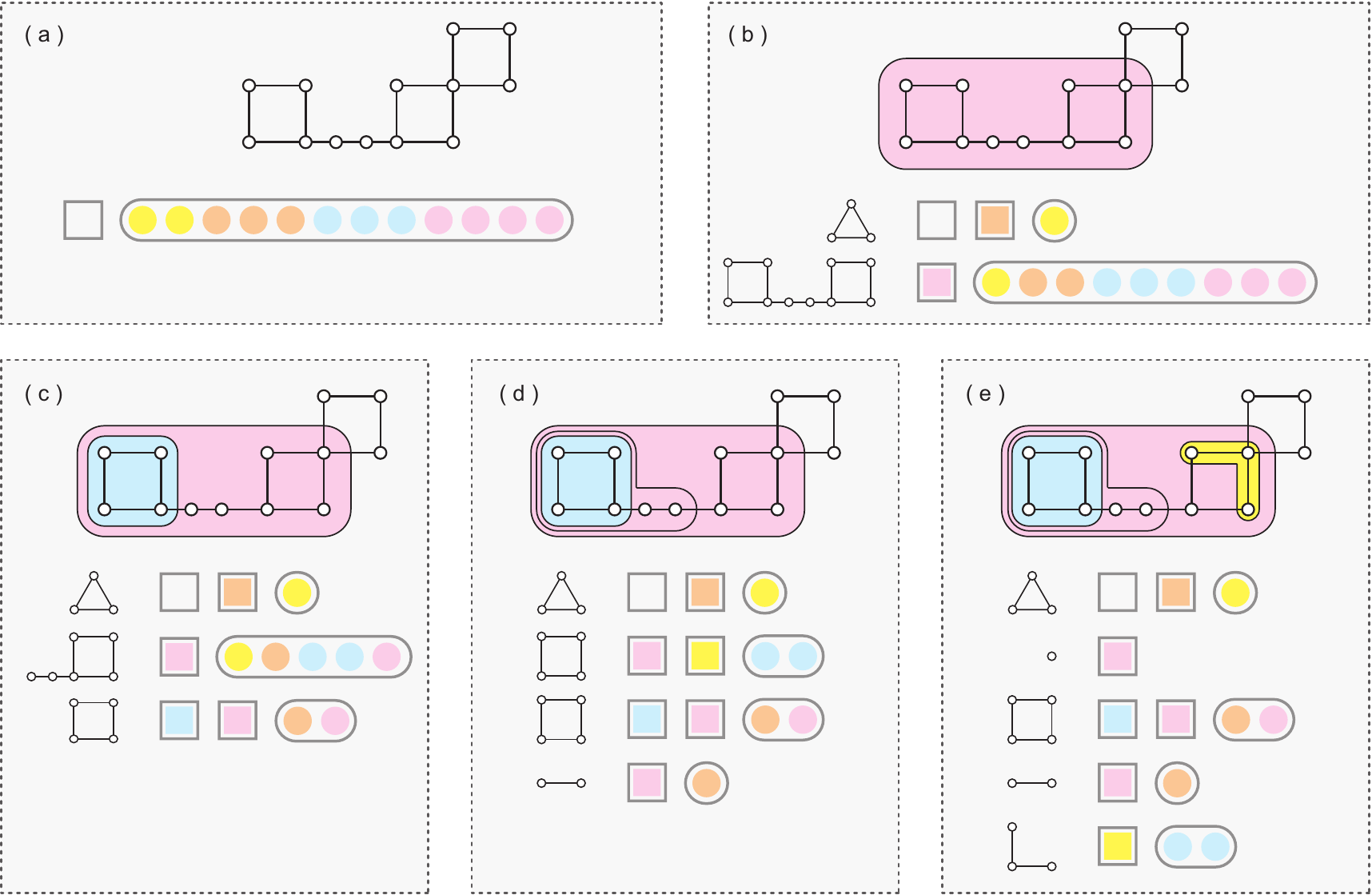}
\caption{Compatibility of color templates.}
\label{f:compatibility}
\end{figure}

\begin{defn}
For a simple graph $G$ with a color palette, the \emph{colorful graph associahedron} $\rKG$ is the collections of posets of color templates $(T, \mathcal O)$ of $G$, with the partial ordering $\prec$ defined above.
\end{defn}

\noindent The proof of Theorem~\ref{t:main} shows this  to be well-defined.

%
%
\section{Colorful Graph Associahedra}  \label{s:colorful}
\subsection{Abstract Polytopes}

We begin with some foundations, culminating in the formulation of one of our main theorems.  The reader is encouraged to explore \cite{ms} for a wealth of information on abstract polytopes.  The elements of a poset $\pos$ are its \emph{faces}, and  two faces $f$ and $g$ are \emph{incident} if $f \preceq g$ or $g \preceq f$.  By convention, there is a face $f_\emptyset$ at rank $-1$ corresponding to the empty set, and faces of rank $0$ are the vertices of $\pos$. In general, the \emph{rank} of face $f$ is defined as $j-2$, where $j$ is the maximum number of faces in any chain of faces $f_\emptyset \prec f_0 \prec f_1 \dots \prec f$.
A \emph{flag} of poset $\pos$ is a totally-ordered set of faces of maximal length, where two flags are \emph{adjacent} if they differ from each other in precisely one face. A poset is \emph{flag-connected} if any two flags $\Phi$ and $\Psi$ can be joined by a sequence of adjacent flags
$$\Phi = \Phi_0, \ \Phi_1, \ \cdots, \ \Phi_n = \Psi,$$ 
and is \emph{strongly flag-connected} with the additional restriction that $\Phi \cap \Psi \subset \Phi_i$, the intersection held constant throughout the sequence. 

\begin{defn}
A poset $\pos$ is an \emph{abstract polytope} of rank $n$ if it satisfies the following: 
\begin{enumerate}
\item $\pos$ contains a least face and a greatest face.
\item Each flag of $\mathcal{P}$ contains exactly $n+2$ faces.
\item $\mathcal{P}$ is strongly flag-connected. 
\item For incident faces $f$ and $h$ (of ranks $j-1$ and $j+1$, respectively), there are precisely two faces $g_1$ and $g_2$ (of rank $j$) such that $f \prec g_i \prec h$.
\end{enumerate}
\end{defn}

This definition captures combinatorial properties that appear naturally for convex polytopes, such as connectivity (requirement 3) and incidence (requirement 4, sometimes called the `diamond' property).  While geometric properties are lost, a wealth of freedom is gained. For example, the \emph{hemi-icosahedron} is an abstract $3$-polytope whose facets tessellate $\RP^2$, with 10 triangles meeting five at each vertex.  Another example is the \emph{$11$-cell}, an abstract 4-polytope tiled by 11 hemi-icosahedron, containing 55 faces, 55 edges, and 11 vertices. 
The following is one our main results, an analog to Theorem~\ref{t:classic-ga}, whose proof is relegated to Section~\ref{s:proofs}:

\begin{thm} \label{t:main}
Let $G$ be a simple graph with $n$ nodes and connectivity $k$, along with a color palette.  The \emph{colorful graph associahedron} $\rKG$ is a collection of simple abstract polytopes of rank $n-1$, consisting of one abstract polytope for each distinct ordering of $k-1$ colors from the color palette.
\end{thm}

In particular, each abstract polytope in the collection corresponds to a distinct assignment of colors to the $k-1$ outer blocks in the universal tube of $G$. As we will show, this assignment of colors precisely determines whether two flags are connected.

\begin{cor} \label{c:main}
For a monochrome color palette, $\rKG$ is the classical graph associahedron.
For a full palette of $n-1$ distinct colors, $\rKG$ consists of $(n-1)!/(n-k)!$ identical abstract polytopes.
\end{cor}

\begin{rem}
For a full palette, an $n$-path (connectivity 1) has no outer blocks, resulting in a unique colorful associahedron (Figure~\ref{f:ksg-path}).   For $n$-cycles (connectivity 2), only the color of the outer block associated to the universal tube must be fixed, resulting in $n-1$ copies of colorful cyclohedra (Figure~\ref{f:ksg-cycle}).  In   \cite{ahos2}, this was addresssed by fixing the color of the long diagonal of the centrally symmetric polygons representing the cyclohedron.
\end{rem}

\subsection{Graphs with Four Nodes}

\begin{figure}[b]
\includegraphics[width=\textwidth]{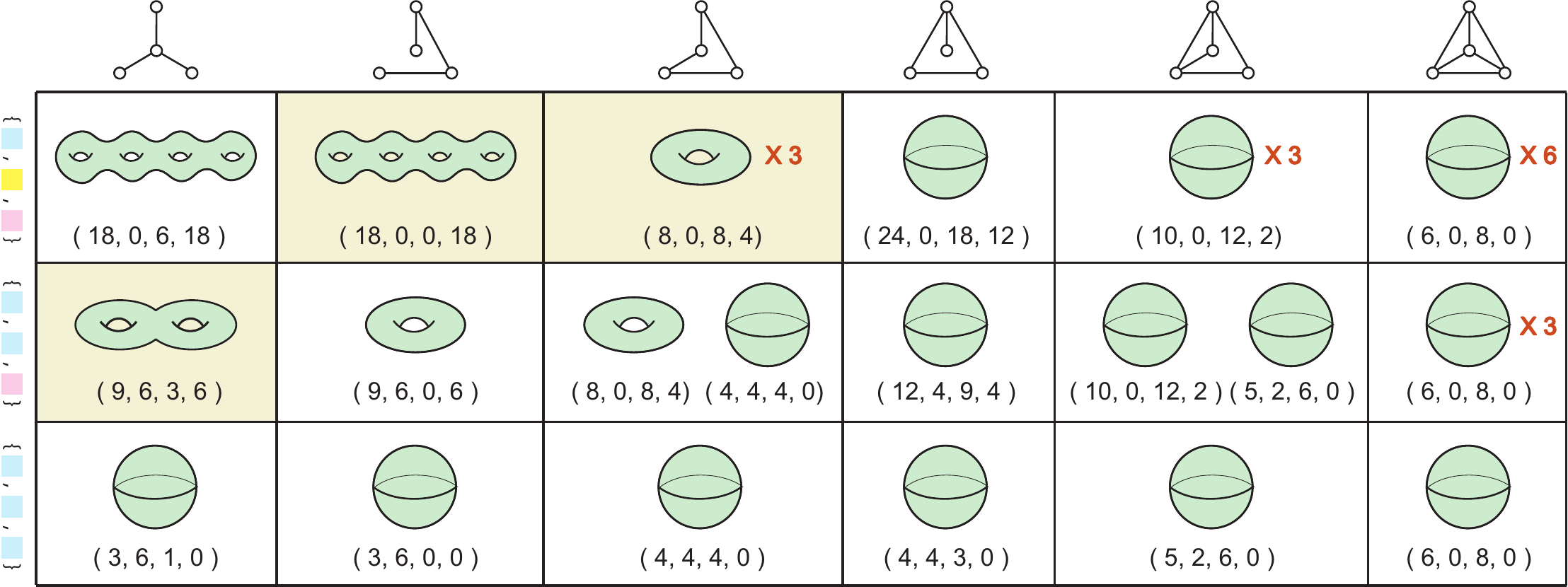}
\caption{Colorful graph associahedra for connected graphs with four nodes.}
\label{f:4nodes-con}
\end{figure}

Consider all the connected simple graphs with four nodes, as displayed by the top row of Figure~\ref{f:4nodes-con}.  The three subsequent rows show the topology of the associated colorful graph associahedra for the three types of color palettes:  full \includegraphics{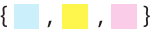}, mixed \includegraphics{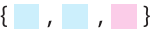}, and monochrome \includegraphics{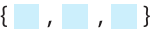}.  All these colorful graph associahedra are (collections of) handlebodies; in particular, the bottom row is simply the classical graph associahedron, topologically equivalent to a 3-ball.

Each entry also has a vector $(k_4,k_5,k_6,k_{10})$, corresponding to the number of squares, pentagons, hexagons, and decagons tiling the respective surface boundaries.  The two highlighted entries in the first row are the colorful associahedron and cyclohedron considered in \cite{ahos2}.  Note that the colorful cyclohedron has three distinct copies, each tiled by the 2-faces in $(8,0,8,4)$, corresponding to the three ways of assigning colors to the outer block.  The bottom row of Figure~\ref{f:4nodes-con} showcases the monochrome case.  In \cite{cdf}, natural cellular surjections  (with certain algebra and coalgebra homomorphims) are shown to exist between  graph associahedra as their underlying graphs are altered. It would be interesting to find analogous maps between colorful graph associahedra as their palettes are transformed.

\begin{exmp}
Consider the shaded entry in the second row of Figure~\ref{f:4nodes-con}, the trivalent graph with color palette \includegraphics{3-mixed}.  The abstract polytope is a solid genus-two surface tiled by 9 squares, 6 pentagons, 3 hexagons, and 6 dodecagons.  The color templates associated to these facets are fully outlined in Figure~\ref{f:trivalent-faces}.

\begin{figure}[h]
\includegraphics[width=.9\textwidth]{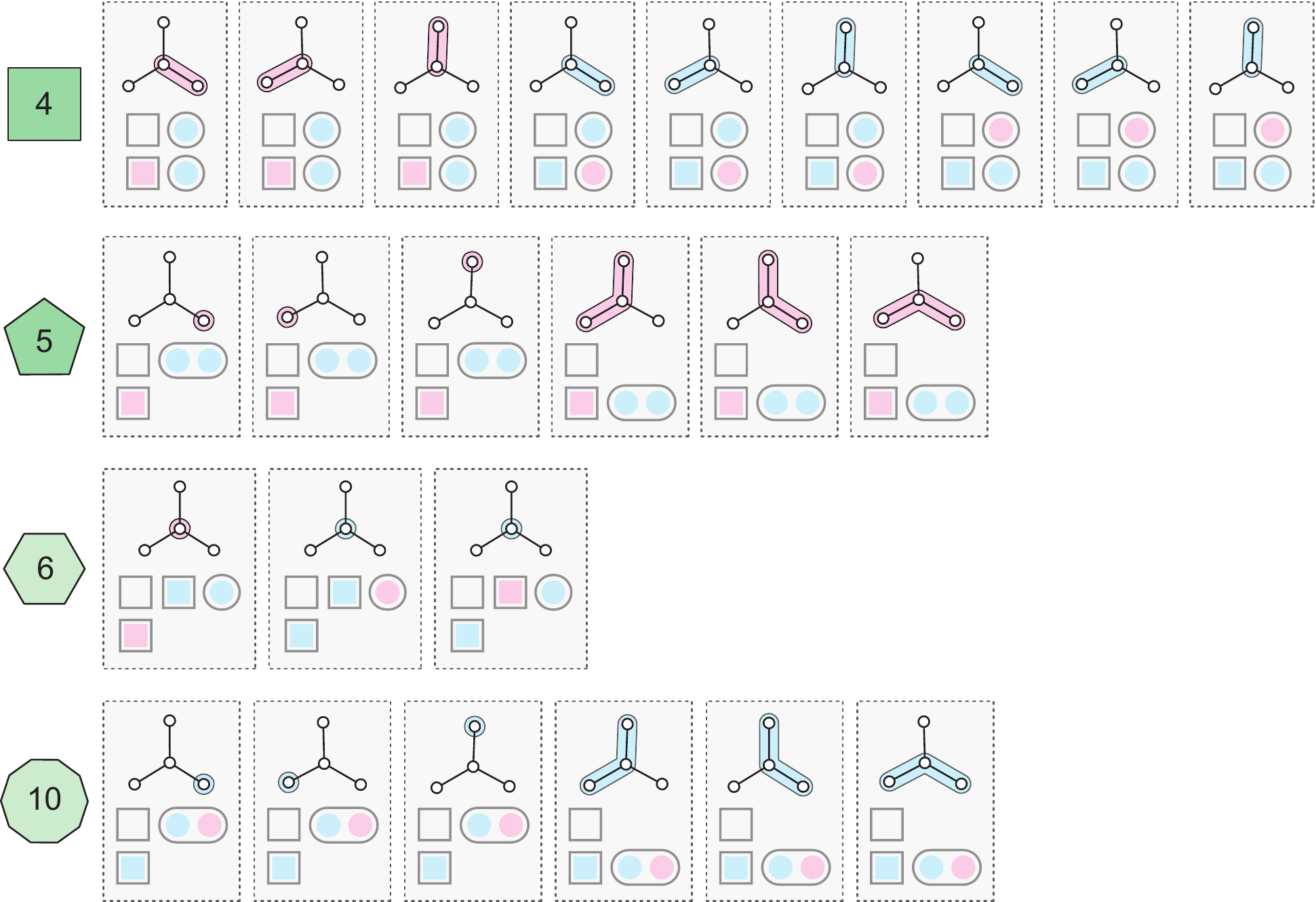}
\caption{Facets of the trivalent graph with a mixed color palette.}
\label{f:trivalent-faces}
\end{figure}

Figure~\ref{f:tiles} examines a part of this genus-two surface up-close, showcasing the gluing of a few facets from Figure~\ref{f:trivalent-faces}.  The pentagons and squares are given a darker shading to help distinguish them from the hexagons and dodecagons.
\end{exmp}

\begin{figure}[h]
\includegraphics[width=1\textwidth]{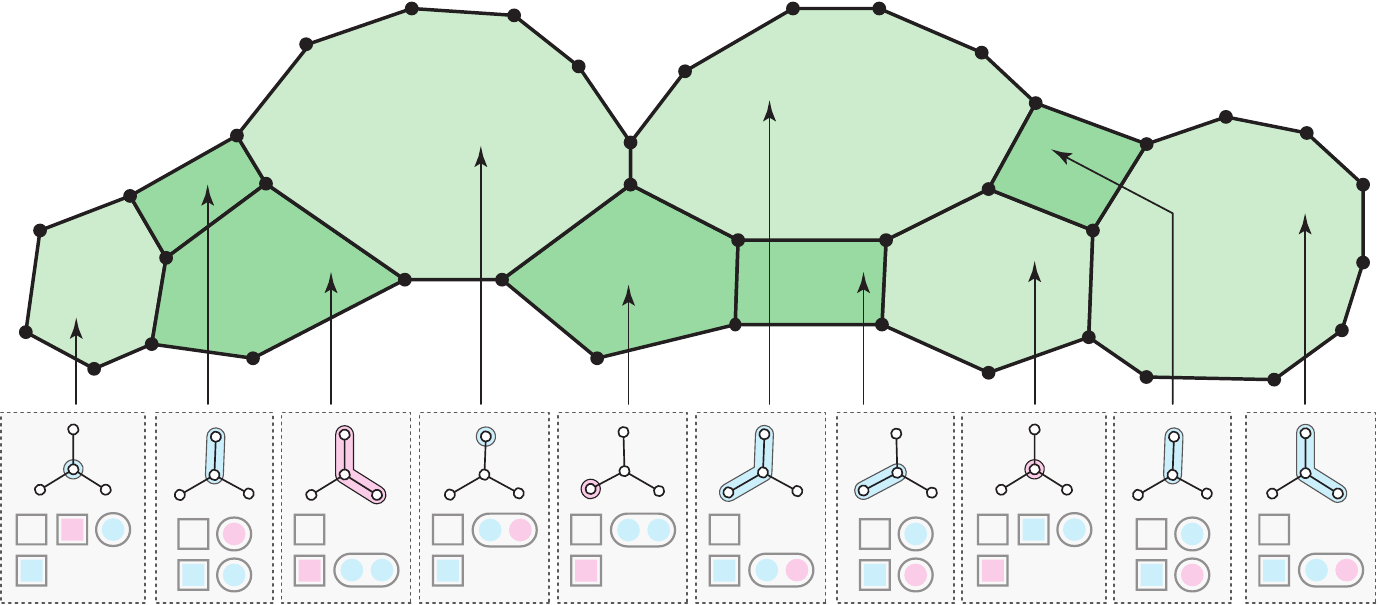}
\caption{A gluing of some facets from Figure~\ref{f:trivalent-faces}.}
\label{f:tiles}
\end{figure}

Corollary~\ref{c:main} shows that every monochrome palette yields the classical graph associahedron.  However, the moment we leave the monochrome world, abstraction sets in.  Even for a palette with just two colors, there are cases in which the colorful graph associahedron is not even a manifold.  Figure~\ref{f:nomanifold}(a) gives the example of a graph with its color palette.  Part (b) is one facet of this graph, corresponding to the color 1-tubing.  But this facet is what we have been discussing in detail, whose colorful graph associahedron is the genus-two handlebody, demonstrating that facets of part (a) are not topological balls.

\begin{figure}[h]
\includegraphics[width=.9\textwidth]{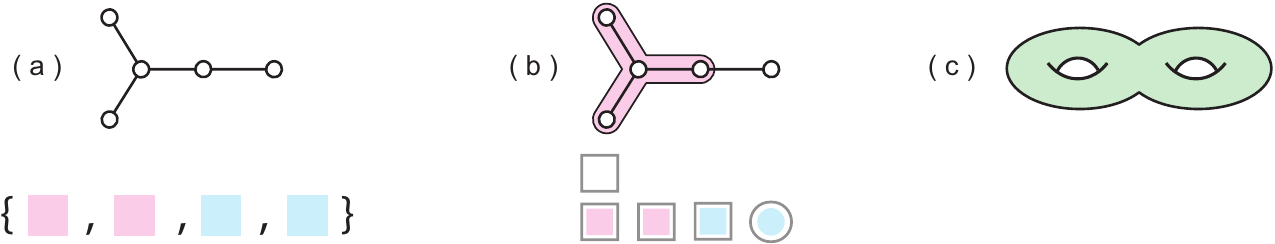}
\caption{(a) Graph with palette, along with (b) a facet and its color template, whose colorful graph associahedron is (c) a genus-two handlebody.}
\label{f:nomanifold}
\end{figure}

\subsection{Product Structures}

We turn our attention from the global nature of $\rKG$ to the local structure of its faces.

\begin{thm} \label{t:prodStruc}
Given a graph $G$, the $j$-faces of $\rKG$ correspond to the color templates of $G$ with an underlying $(n-j-1)$-tubing. Moreover, the $j$-face associated to $(T,\mathcal O)$, where $T$ is the set  $\{t_1, \dots ,t_{n-j-1}\}$ of tubes, is combinatorially equivalent to 
$$\rK {\rc {t_1} T} \times \dots \times \rK{\rc {t_{n-j-1}} T} \times \rK{\rc {G} T}\,.$$
\end{thm}

\begin{proof}
The proof is akin to that of the product structure of facets in $\KG$ presented in \cite{cd}. As per Proposition~\ref{p:core}, there is a natural bijection from each tube $t_i$ of $G$ to the universal tube of $\rc {t_i} T$.  Moreover, the structure of the color blocks associated to $t_i$ in $T$ is the same as the structure of those associated to the universal tube of $\rc {t_i} T$. This gives rise to the following mapping: For a color template $(T,\mathcal O)\in \rKG$,  where $T$ is the set $\{t_1, t_2, \cdots ,t_{n-j-1}\}$ of tubes and $o_i$ is the collection of color blocks associated to $t_i$, define 
$$(T,\mathcal O) \ \mapsto \ ((\rc {t_1} T, o_1), \dots, (\rc {t_{n-j-1}} T,o_{n-j-1}), (\rc {G} T,o))\,.$$
It is straight-forward to verify that this is a bijection.
\end{proof}

This theorem allows us to look at any $j$-face and, based on its color template, completely determine its structure. However, since the cores are dependent on $G$, the structure of each $j$-face is a globally dependent property.

\begin{exmp}
Figure~\ref{f:product}(a) depicts a color template with an underlying 2-tubing. The core of the pink tube is a path with two nodes, which gives rise to a line segment as its colorful graph associahedron. The core of the blue tube is a path with three nodes, which (with the monochrome pink palette) yields a pentagon.  Since the universal tube is simply a node, the face of the given 2-tubing is a pentagonal prism, shown in (b), as  Theorem~\ref{t:prodStruc} claims. Part (c) lists the seven tubings associated to each face of the prism.
\end{exmp}

\begin{figure}[h]
\includegraphics[width=\textwidth]{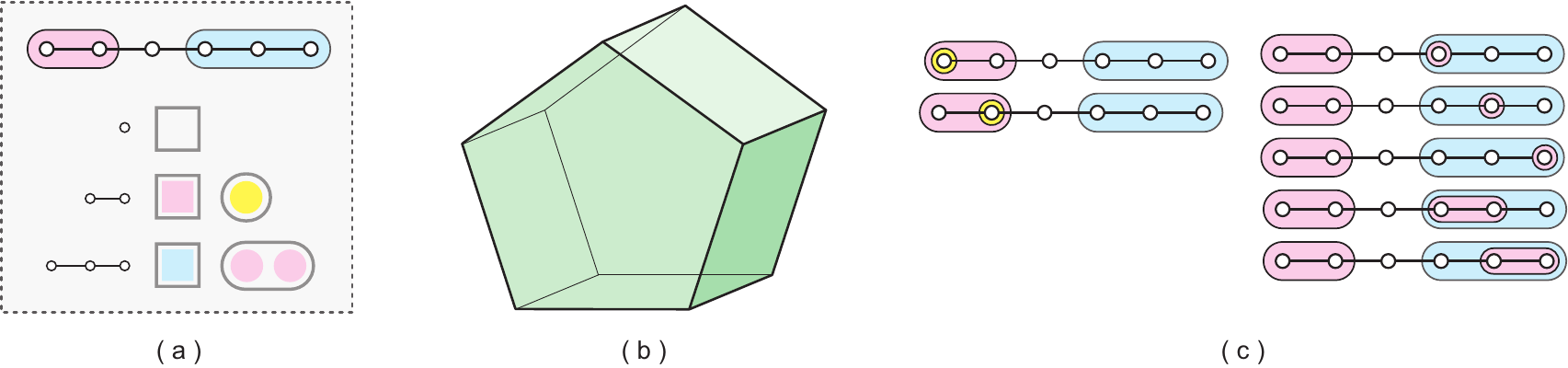}
\caption{A face with a product structure.}
\label{f:product}
\end{figure}

%
%
\section{Disconnected Graphs}  \label{s:disconnect}
\subsection{Modifications}

We extend the definitions and results for colorful graph associahedra to disconnected graphs.  Let $G$ be a graph with connected components $G_1, \dots, G_m$.   As for the classic graph associahedra  \cite{cdf}, any tubing of $G$ cannot contain all of the tubes $\{G_1, \dots, G_m\}$.  Thus, for a graph with $n$ nodes, a maximal tubing  still contains exactly $n-1$ tubes.   Based on this, we make a few minor alterations to existing definitions:

\begin{enumerate}
\item {\bf Universal Modification:}
The universal tubes are now the connected components of $G$ (and no longer the entire graph). As before, the universal tubes are included in the color template, but now they can also take on a color. 
\item {\bf Template Modification:}
A color template is the pair $(T,\mathcal O)$ as before, but now $\mathcal O$ is a partition of the color palette amongst the tubes of $T$ \emph{and} the universal tubes $G_i$. Here, an additional \emph{universal block} is created to store the colors that will be used on the universal tubes. 

\item {\bf Compatibility Modification:}
Compatibility of $(T, \mathcal O) \prec (T', \mathcal O')$ is as before, with the following addendum:  For a universal tube $G_i$ of $G$, with tubings $T'$ and $T = T' \cup \{G_i\}$,
\begin{enumerate}
\item In $\mathcal O$, the ordered partitions for all tubes in $T \setminus \{G_i\}$ must be the same as in $\mathcal O'$.
\item In $\mathcal O$, the inner and outer blocks of $G_i$ are colored as dictated by $\mathcal O'$, but the color of $G_i$ is selected from the universal block of $\mathcal O'$.
\end{enumerate}
\end{enumerate}

\begin{exmp}
Consider Figure~\ref{f:dis-compact}, the example of a disconnected graph with three components: a 4-cycle, a 3-path, and a node. As with connected graphs, the color template in (a) consists of a tubing and partition of the color palette. Each component has color blocks assigned according to its connectivity. We now have  the addition of a universal block, designating colors (pink, yellow) that can be used for the universal tubes. As universal tubes are added, their color will be taken from this block. 

\begin{figure}[h]
\includegraphics[width=\textwidth]{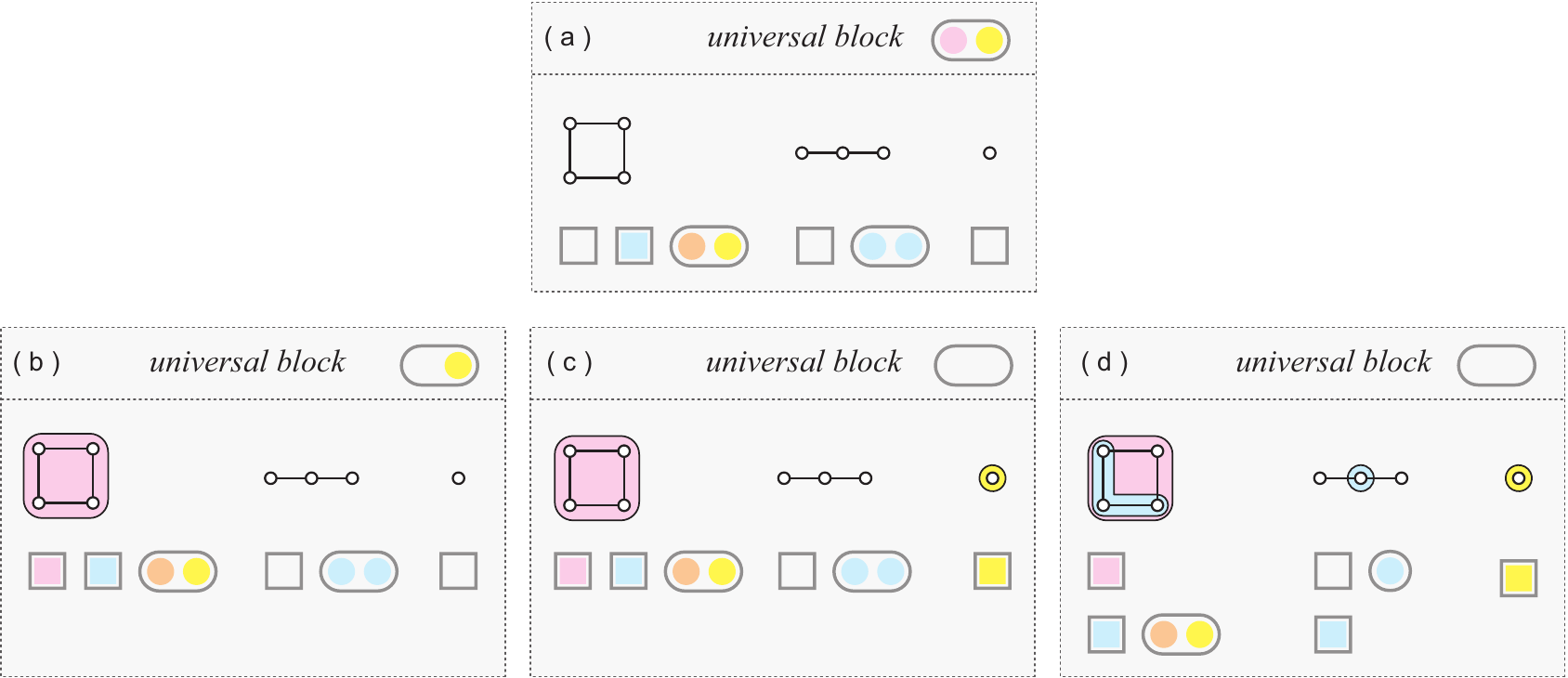}
\caption{Compatibility of disconnected colored graphs.}
\label{f:dis-compact}
\end{figure}

The second row in Figure~\ref{f:dis-compact} represents compatible color tubings achieved through the iterative addition of tubes to (a). A universal tube is added in (b), given an available color (pink) from the the universal block. The colors for its inner and outer blocks remain those designated by the template in (a). Another universal tube is added in (c), using the only available color (yellow) in the universal block. At this point, no colors remain in the universal block, and likewise, the maximal number of universal tubes has been reached.  Non-universal tubes may still added, as in (d), whose color blocks are determined according to  normal compatibility rules.
\end{exmp}

Consider the simplest example of a disconnected graph $G[n]$, the null graph on $n$ nodes.  Proof of the following is provided in Section~\ref{s:proofs}:

\begin{prop} \label{p:null}
Given $G[n]$ with a color palette, its colorful graph associahedron, denoted as $\rKGn n$, is a connected simple abstract polytope. 
\end{prop}

\begin{exmp}
Figure~\ref{f:torus} displays the net of the boundary  of $\rKGn 4$, a solid torus tiled by hexagons.  Each hexagon is a copy of $\rKGn 3$  corresponding to a color tube about one of its four nodes. The color template for a hexagon and one of its vertices is provided.  In general, each facet of $\rKGn n$ corresponds to a unique color and node combination, and  is adjacent to precisely those facets that correspond to both a different node and color.
\end{exmp}

\begin{figure}[h]
\includegraphics[width=\textwidth]{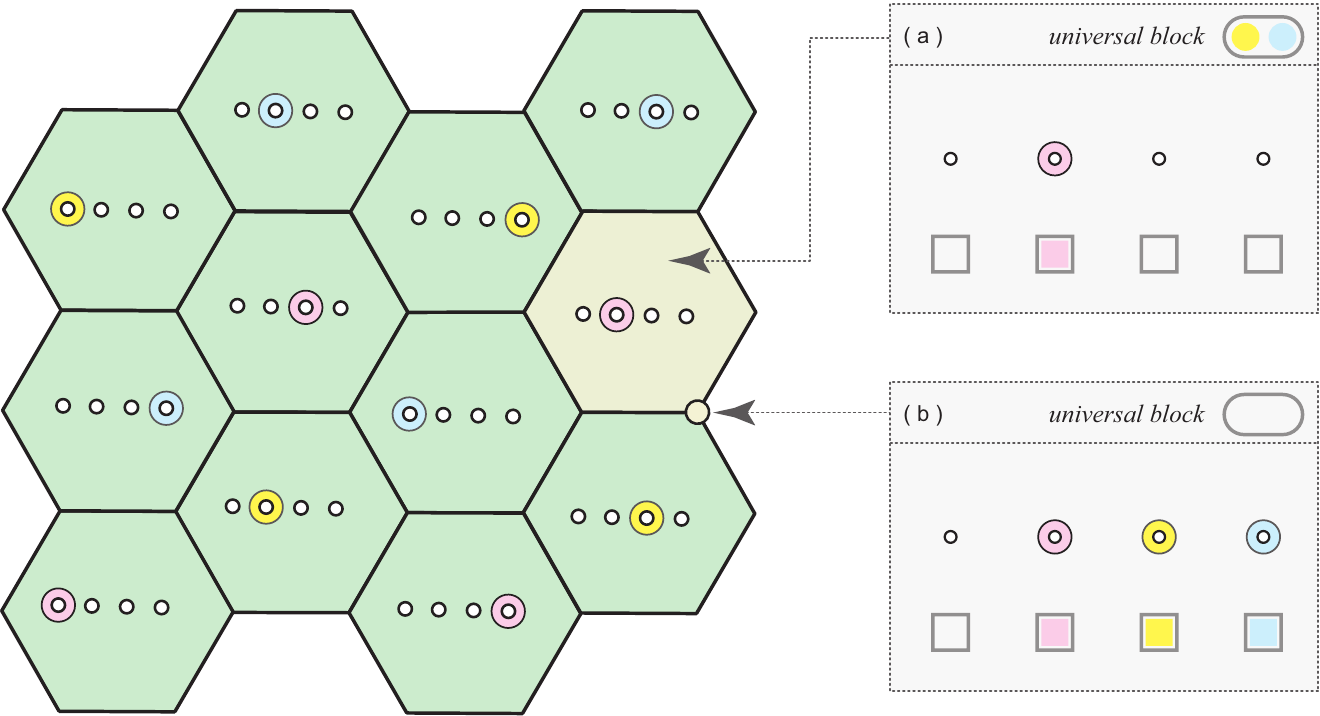}
\caption{The colorful graph associahedron $\rKGn 4$.}
\label{f:torus}
\end{figure}

\subsection{Structures}
We turn our attention to general disconnected graphs. In particular, we show that the colorful graph associahedron has a product structure given by its connected components.

\begin{thm}  \label{t:discon}
Let $G$ be a simple graph with connected components $G_1, \dots, G_m$, along with a color palette. The colorful graph associahedron $\rKG$ is a collection of simple abstract polytopes. Moreover, each ordered partition of the palette into sets of sizes $|G_1|-1, \dots, |G_m|-1, m-1$ corresponds to a unique subset of abstract polytopes of $\rKG$ given by 
\begin{equation} \label{e:product}
\rKG_1  \times \cdots \times \rKG_m \times \rKGn m\,,
\end{equation}
where the palette for $\rKG_i$ is given by the partition class of size $|G_i|-1$ and the palette for $\rKGn m$ is given by the partition class of size $m-1$.
\end{thm}

\begin{proof}
We create an order-embedding bijection $\phi$ from the faces of the form in \eqref{e:product}, satisfying the conditions outlined above, to the faces of $\rKG$. Given a face described by the ordered tuple 
$$((T_1, \mathcal{O}_1), \ \dots,  \ (T_{m+1}, \mathcal{O}_{m+1}))\,,$$
define $\phi$ as follows: The first $m$ arguments follow the natural mapping: for each face $(T_i,\mathcal{O}_i)$ of $\rKG_i$, the tubing $T_i$ is constructed on component $G_i$ in $G$. The color order $\mathcal{O}_i$ is not changed and is simply associated to the corresponding tubes on $G$. 

The final argument uses a different mapping: Given the face $(T_{m+1}, \mathcal{O}_{m+1})$ of $\rKGn m$, a tube is constructed around $G_i$ if and only if $T_{m+1}$ includes a tube around the $i$-th node of $G[m]$. As before, the color order $\mathcal{O}_{m+1}$ is not changed and is simply associated to the corresponding tubes on $G$.  This holds even for the universal block. Note that $\phi$ is an order-embedding bijection and maintains the validity of the tubings. It is straight-forward to verify this is an order isomorphism.
By \cite[Theorem A]{gh}, since each term in \eqref{e:product} is a collection of abstract polytopes, $\rKG$ is a collection of abstract polytopes. Simplicity follows immediately.
\end{proof}

\begin{cor} \label{c:discon}
Let $\Delta_i$ be the $i$-simplex. For a monochrome color palette, $\rKG$ is isomorphic to
$$\KG_1 \times \dots \times \KG_m \times \Delta_{m-1}\,.$$
For a full palette of $n-1$ distinct colors, $\rKG$ is isomorphic to
$$\rKG_1 \times \dots \times \rKG_m \times \rKGn m\,,$$
consisting of
$${n-1 \choose m-1} \cdot {n-m \choose |G_1| -1} \cdot {n-m - |G_1| +1 \choose |G_2| -1} \ \cdots \ {|G_m| -1 \choose |G_m|-1}$$ 
copies of identical polytopes.
\end{cor}

\begin{proof}
The monochrome results follows from the classical graph associahedron \cite[Theorem 2]{cdf}. For the full palette, the enumeration is obtained from the appropriate partitions of the colors amongst the connected components.
\end{proof}

\begin{rem}
It follows that for a multi-component graph $G$, the colorful graph associahedron $\rKG$ is connected only when the color palette is monochrome or when $G$ is the null graph.
\end{rem}

\begin{figure}[b]
\includegraphics[width=\textwidth]{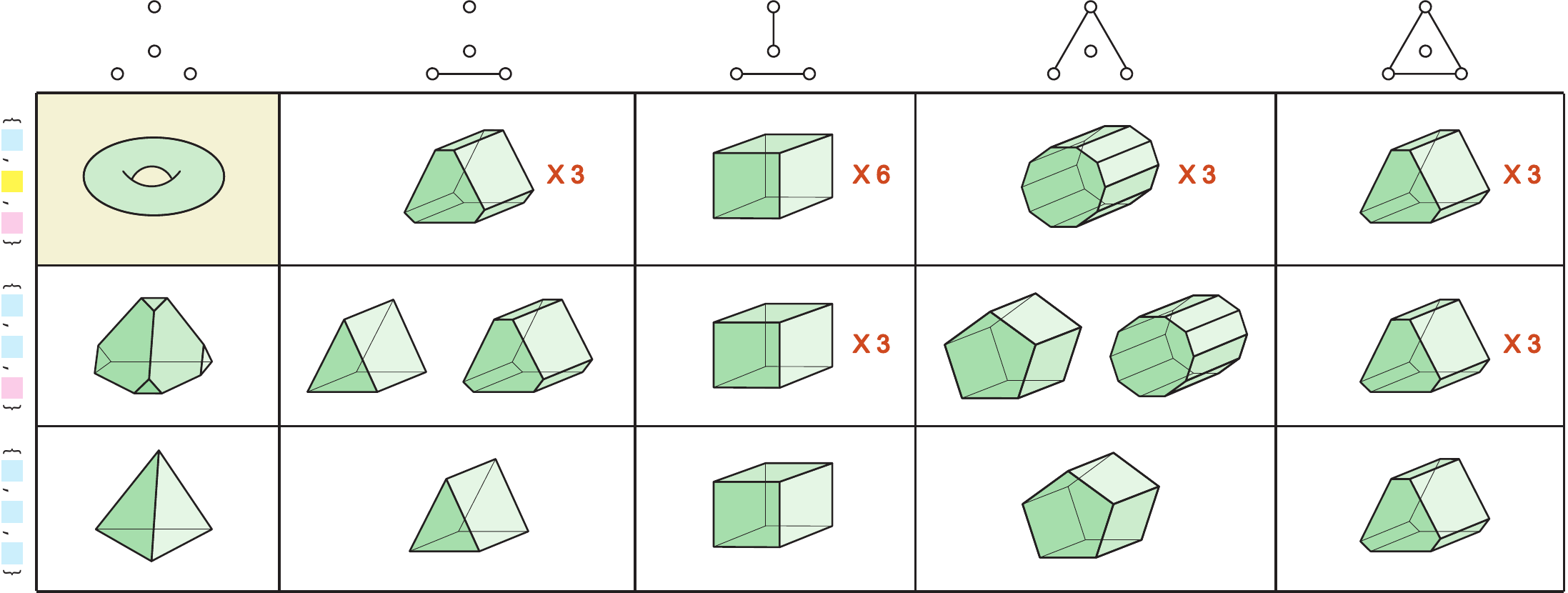}
\caption{Colorful graph associahedra for disconnected graphs with four nodes.}
\label{f:4nodes-discon}
\end{figure}

\begin{exmp}
Figure~\ref{f:4nodes-discon} depicts the colorful graph associahedra for disconnected graphs with four nodes. Similar to Figure~\ref{f:4nodes-con}, the three subsequent rows show the geometry of colorful graph associahedra associated to these graphs, for the three types of color palettes:  full \includegraphics{3-full}, mixed \includegraphics{3-mixed}, and monochrome \includegraphics{3-mono}. The number of connected components decreases monotonically moving down the rows: as distinct colors are removes, the number of ways to partition the palette amongst the components and universal block decreases.

While the shaded entry of Figure~\ref{f:4nodes-discon} is a solid torus (detailed in Figure~\ref{f:torus}), the rest are convex polytopes. Convexity is common here because any graph with two components contains the interval $\rKGn 2$ in its product. But as the number of nodes and connected components of $G$ increases, the polytope structure of $\rKG$  will move from convex to abstract.
\end{exmp}

\subsection{Regularity}

For traditional polytopes, regularity requires that the symmetry group be flag-transitive. However, abstract polytopes do not necessarily have a symmetry group and a geometry. By weakening the notion of regularity and only requiring flag-transitivity of the automorphism group, the definition of regularity can be extended to abstract polytopes \cite{ms}. For the remainder of this section, we consider only colorful graph associahedra with \emph{full} palettes.

\begin{defn}
Let $P$ be an abstract polytope. Then, $P$ is \emph{(combinatorially) regular} if its automorphism group $\Gamma(P)$ is flag-transitive. Equivalently, for any two flags $\Phi$ and $\Psi$, there exists an element $\sigma \in \Gamma(P)$ such that $\sigma (\Phi) = \Psi$.
\end{defn}

Given the high degree of symmetry, it is not surprising that $\rKGn n$ is regular. In fact, it mostly turns out to be the only example. Let $P_i$ be the path and $K_i$ be the complete graph on $i$ nodes.

\begin{lem} \label{l:connregular}
Let $G$ be a simple connected graph, along with a full color palette. If $\rKG$ is regular, then $G$ is either  $K_1$, $K_2$, $K_3$, or $P_3$.
\end{lem}

\begin{proof}
First, observe that the colorful graph associahedra of $K_1$, $K_2$, $K_3$, and $P_3$ are all regular. 
The following argument examines only tube compatibility; it doesn't specify color orders. This is permitted because all facets with the same underlying tubing are isomorphic, a consequence of the full color palette. 

If $G$ contains more than one node, let $x$ and $y$ be adjacent nodes. Since $\rKG$ is  flag transitive, a facet corresponding to tube $\{x\}$ must be isomorphic to a facet corresponding to tube $\{x,y\}$. Hence
$$\rK{\rc {G} {\{x\}}} \ \cong \ \rKGn 2 \times \rK{\rc {G} {\{x,y\}}}.$$
Since $\rKGn 2$ is an interval, there must exist tubes $t_1$ and $t_2$ of $\rc {i} {\{x\}}$ which are compatible with all other tubes except each other. Removing $t_1$ from $\rc {G} {\{x\}}$  creates a number of connected components, all of which are adjacent to (and thus not compatible with) $t_1$. Hence, $\rc {G} {\{x\}}-t_1$ consists of only one component, which we denote as $t_2$.
Since $\rc {G} {\{x\}}$ is connected, there is an edge $\{v_1,v_2\}$ with $v_1\in t_1$ and $v_2\in t_2$. Suppose $t_1$ or $t_2$ has more than one node. Then, the  $\{v_1,v_2\}$ is a valid tube but is not compatible with $t_1$ or $t_2$, contradicting the definition of $t_1$ and $t_2$. We conclude that $G$ has at most 3 nodes; as it is connected, the result follows.
\end{proof}

\begin{prop}
\label{p:regular}
Let $G$ be a simple graph with $n \geq 5$ nodes, along with a full color palette. Then $\rKG$ is a regular abstract polytope if and only if $G$ is the null graph $G[n]$.
\end{prop}

\begin{proof}
In the backwards direction, assume $G$ is the null graph.  Let $\Gamma_n$ denote the permutation group on $n$ elements. The flags of $\rKGn n$ are in bijection with the elements $(s, t)$ of $\Gamma_{n-1} \times \Gamma_{n-1}$: 
\begin{enumerate}
\item $s$ determines the labeling of the nodes of $G[n]$, determining the tube ordering in the flag. 
\item $t$ determines the assignment of colors to tubes.
\end{enumerate}
Given flags $(s_1, t_1)$ and $(s_2, t_2)$ of $\rKGn n$, there exist permutations $p, q$ such that $ps_1 = s_2$ and $qt_1 = t_2$, with $p$ relabeling nodes and $q$ relabeling colors. Hence $(s,t) \mapsto (ps, qt)$ is the desired automorphism, ensuring the regularity of $\rKGn n$.

In the forwards direction, assume  $\rKG$ is regular for a simple graph $G$ with $m$ connected components, $G_1, \dots, G_m$. Thus $\rKG_i$ is also regular, and by Lemma~\ref{l:connregular}, $G_i$ must be one of $K_1$, $K_2$, $K_3$, or $P_3$.  Since the palette is full, Corollary~\ref{c:discon} claims that $\rKG$ is isomorphic to
\begin{equation} \label{e:cube}
\rKG_1 \times \dots \times \rKG_m \times \rKGn m\,.
\end{equation}

In \cite[Section 5]{gh}, it is proven that all abstract polytopes have a unique prime factorization under the Cartesian product. Moreover, if an abstract polytope is regular, then it is either prime or has a prime factorization that is isomorphic to a product of intervals. 
We apply these results to our situation.

If $\rKG$ is prime, then either $G$ is connected or  is the null graph $G[n]$. Since $n \geq 5$, Lemma~\ref{l:connregular} guarantees the impossibility of the former, meaning \eqref{e:cube} is isomorphic to a product of intervals.
But component $\rKG_m$ is isomorphic to a product of intervals only when $m=1$ or $m=2$. The former case has been discussed. For the latter, since $n \geq 5$, one component must be $K_3$ or $P_3$. But since neither is isomorphic to a product of intervals, $G$ must be the null graph $G[n]$.
\end{proof}

%
%
\section{Proofs of the Main Results} \label{s:proofs}
\subsection{Exchange Graph}

For the duration of this paper, let $G$ be a simple connected graph with $n$ nodes and connectivity $k$, along with a color palette.  Thus far, we have described $\rKG$ as a collection of posets of color templates, highlighting the structure of $\rKG$ and its relationships to connectivity and color tubings. Drawing inspiration from~\cite{ahos1}, we provide an alternate definition of $\rKG$ starting from its 1-skeleton.

\begin{defn}
The nodes of the \emph{exchange graph} $\eG$ correspond to maximal color tubings of $G$, where two nodes are adjacent if their tubings differ at precisely one tube. Note that the color of the differing tubes must be the same, and by Proposition~\ref{p:maxtubing}, the tube type (inner versus outer) must also be the same.
\end{defn}

By Proposition~\ref{p:maxtubing}, each maximal tubing of $G$ contains precisely $k-1$ outer tubes: a unique outer tube containing $i$ nodes, one for each $n-k < i < n$.  We say two maximal color tubings are \emph{color-matched} if the matching pairs of outer tubes (with $i$ nodes) from each tubing are identically colored.

\begin{prop} \label{p:connect}
Two maximal color tubings of $G$ are color-matched if and only if they are nodes in the same connected component of $\eG$.
\end{prop}

\begin{proof}
The discussion above guarantee that if $T_1$ and $T_2$ are connected by a path, then the outer tubes of  $T_1$ and $T_2$ with the same number of nodes will have identical colors. Thus, they must be color-matched, concluding one direction.

For the other direction, we show there exists a path between two color-matched nodes $T_1$ and $T_2$ of $\eG$.  Denote $mono(T_*)$ to be the monochrome version of $T_*$. Since the 1-skeleton of the classical associahedron $\KG$ is connected, there exists a path between $mono(T_*)$ and any monochrome maximal tubing $U$ of $G$.  This naturally lifts to a path in $\eG$, between $T_*$ and some node $U_*$ where $mono(U_*) = U$.
To prove there exists a path between $U_1$ and $U_2$ with $mono(U_1)=U=mono(U_2)$ (concluding the proof), we show that the colors of any two inner tubes in a coloring of $U$ can be exchanged. Repeatedly performing color exchanges yields a path between the desired tubings. 

To start, notice that the inner tubes of any maximal tubing reside on a subgraph of $G$ with connectivity 1. Since we only need to show that the colors of any two inner tubes can be exchanged, assume that $G$ has connectivity 1. 

\begin{claim}
If $G$ has at least 3 nodes, there is a node $w$ such that $G-w $ has connectivity 1.  
\end{claim}

\begin{proof}
Since $G$ has connectivity 1, there exists a node $v$ whose removal disconnects $G$ into multiple components, $G_1, G_2, \cdots ,G_m$. (If each component contains only one node, there will be at least two components and the single node of $G_1$ is an appropriate choice for $w$.)  Without loss of generality, say component $G_1$ contains more than one node. Define $H_1$ to be the induced subgraph on the nodes of $G_1$ together with $v$. Construct the spanning tree of $H_1$ and choose any leaf of this tree (other than $v$) and call it $w$.   Thus, $H_1-w$ is connected, and thus, so is $G-w$. However, since $G_1$ contains more than one node, $(G-w)-v$ is not connected.
\end{proof}

We proceed by induction on the number of nodes of $G$. Since the only graph of connectivity 1 with less than 3 nodes is the path on 2 nodes (for which $\eG$ is connected), consider a graph $G$ with $n \geq 3$ nodes. We now specify the tubing $U$: Pick a maximal tubing that includes the tubes $t:=G-w$ and $t':=G_m$, in addition to the tubes $G_1-w, G_2, \dots, G_{m-1}$. The remaining tubes can be chosen arbitrarily. By the inductive hypothesis, the colors of any two tubes inside of $t$ can be exchanged. Figure~\ref{f:exchange} depicts the algorithm which exchanges the colors of $t$ and $t'$, completing the proof.\footnote{This algorithm appears in \cite[Figure 1]{orbit} as walking halfway around the 10-gon in Figure~\ref{f:ksg-path} above.}
\end{proof}

\begin{figure}[h]
\includegraphics[width=\textwidth]{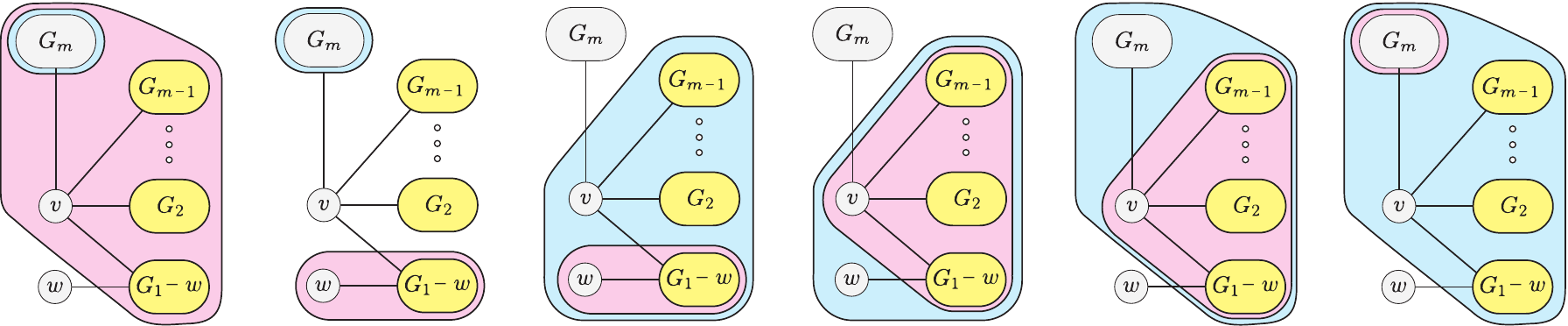}
\caption{Exchanging the colors of two tubes.}
\label{f:exchange}
\end{figure}

\subsection{Poset Isomorphism}

As the vertices of $\rKG$ are the nodes of $\eG$, it is possible to redefine the faces of $\rKG$ as subsets of these nodes, providing an alternate poset formulation of $\rKG$. In particular, we later show that $\eG$ is the 1-skeleton of $\rKG$.

\begin{defn}
Given a tubing $T$ of $G$, a path in $\eG$ \emph{preserves} $T$ if for each node in the path, the associated maximal tubing contains $T$.
\end{defn}

\begin{defn}
Let $\PeG$ be a collection of posets, one for each connected component of $\eG$. Each poset consists of elements and a partial order, as outlined below:
\begin{enumerate}
\item {\bf Faces:} For the faces of rank $j=0,1,2, \dots n-1$, define the $j$-face $(T,v)$ to be the set of nodes in the component of $\eG$ that are reachable from $v$ via a path that preserves the $(n-j-1)$-tubing $T$. Append a unique face $f_{\emptyset}$ of rank -1 that has $f_{\emptyset} \prec f_0$ for all faces $f_0$ of rank 0.
\item {\bf Partial order:} Given faces $f$ and $h$, we say that $f\prec h$ if and only if the node set of $f$ is contained in the node set of $h$.
\end{enumerate}
\end{defn}

\begin{rem}
The representation of the face $(T,v)$ is not unique: $(T,v) = (T',v')$ if and only if $T = T'$ and $v$ and $v'$ are connected by a path preserving $T$. Consider two faces $f=(T_f,v_f)$ and $h=(T_h,v_h)$, where $f\prec h$. Notice that $h$ can be equivalently represented as $(T_h,v_f)$ and $T_h$ must consist of a subset of the tubes of $T_f$. 
\end{rem}

By the definition of $\PeG$, there is one poset for each connected component of $\eG$.  Each $(n-1)$-face  $(T,v)$  has $|T|=0$ and corresponds to all of the nodes in a connected component of $\eG$. Each 0-face $(T,v)$ has $|T|=n-1$ and corresponds to the single node $v$ of $\eG$. All color-matched faces of rank 0 are incident to the same least face $f_{\emptyset}$.
In a similar manner, there is one poset in the collection $\rKG$ for each distinct assignment of colors to the $k-1$ outer blocks in the universal tube of $G$: for two color templates to be comparable, they must be compatible with the same color template for the 0-tubing.

\begin{prop} \label{p:equivdefn}
There is a bijection between the collections $\rKG$ and $\PeG$. Moreover, each poset of $\rKG$ and its corresponding poset in $\PeG$ are order-isomorphic.
\end{prop}

\begin{proof}
Each poset of $\rKG$ corresponds to a distinct assignment of colors to the outer blocks of the universal tube of $G$, reimagined as an assignment of colors to the outer tubes of all maximal tubing. In fact, this assignment is precisely what distinguishes each connected component of $\eG$. The bijection between the collections $\rKG$ and $\PeG$ is immediate.

Consider a poset $\PeGi$ in the collection $\PeG$ and its corresponding poset $\rKGi$ in $\rKG$. The 0-faces of $\PeGi$ and $\rKGi$ both correspond to the color-matched maximal color tubings of $G$, yielding an immediate bijection between 0-faces. Moreover, this bijection can be extended to faces of higher rank. Let $f=(T,\mathcal O)$ be a face of $\rKGi$ with positive rank, and let $f_0$ be a 0-face incident to $f$. Suppose $f_0$ corresponds to the node $v_f\in\eG$. Let $h_0$ be an arbitrary 0-face associated to some $v_h\in\eG$. Apply Proposition~\ref{p:connect} to each core graph of $T$. There is a path from $v_h$ to $v_f$ that preserves $T$ if and only if 
$h_0\prec f$. From this, an order-embedding bijection between the faces of $\rKGi$ and $\PeGi$ follows. Each face $f$ of $\rKGi$ is mapped to the set of nodes in $\eG$ which correspond to precisely the set of maximal tubings represented among its 0-faces. 
\end{proof}

\subsection{Flags}

We consider the flags of $\PeG$. To simplify notation, we exclude the face $f_{-1}$ from the flags. By the discussions above, each flag is of the form 
$$\Phi \ = \ \{(T_0,v), \ \dots , \ (T_{n-1},v)\}\ = \ (\mathcal{T}, v)\,,$$ 
where $\mathcal T$ is the maximal nested family of subsets $T_0\subset T_1 \subset \cdots \subset T_{n-1}$. 

\begin{prop}
\label{p:diamond}
Each poset in the collection $\PeG$ satisfies the diamond condition.
\end{prop}
\begin{proof}
Let $\PeGi$ be a poset of $\PeG$. Each flag $\Phi = (\mathcal{T}, v)$ of $\PeGi$ has exactly one $j$-adjacent flag $\Phi^j$, differing from $\Phi$ in the $j$-th face. When $j\geq1$, the tubing $T_{j-1}$ is a subset of $n-j$ tubes and the tubing $T_{j+1}$ is a subset of $n-j-2$ tubes. Hence, there are precisely two $(n-j-1)$-subsets $T$ such that $T_{j+1}\subset T \subset T_{j-1}$, one of which is $T_j$; this leaves precisely one $j$-adjacent face. When $j=0$, one can appeal to the behavior of the monochrome version of $T_0$.
\end{proof}

\begin{prop}
\label{p:strongflagconn}
Each poset in the collection $\PeG$ is strongly flag-connected.
\end{prop}

\begin{proof} 
Our approach builds on \cite[Lemma 3.1]{ahos1} and adapts their technique to our more general poset.
Let $\PeGi$ be a poset of $\PeG$. For a polytope of rank $n-1$, a $\textit{vertex-figure}$ is the set of faces $f_{n-1}/f_0=\{h\hspace{.1cm}|\hspace{.1cm} f_0 \prec h \prec f_{n-1}\}$, where $f_0$ is a face of rank 0 and $f_{n-1}$ is a face of rank $n-1$. If $f_0=(T,v)$ for some tubing $T$ and node $v$, then all faces of the vertex-figure can be written the form $(T',v)$, where $T'\subset T$. Moreover, they are in bijection with the subsets of tubes of $T$. Additionally, for any two faces in the vertex-figure, $(T',v) \prec (T'',v)$ if and only if $T'' \subset T'$. Hence, the vertex-figure of $f_0$ is isomorphic to the Boolean lattice on $T$, and thus strongly flag-connected. 

Now, let $\Phi$ and $\Lambda$ be flags in the same poset of $\PeGi$ with 
$$\Phi \ = \ \{(T_0,v), \ \dots , \ (T_{n-1},v)\}\ = \ (\mathcal{T}, v) \ \ \textup{and} \ \ \Lambda \ = \ \{(U_0,w), \ \dots , \ (U_{n-1},w)\}\ = \ (\mathcal{U}, w)\,.$$ 
We show that there exists a sequence of adjacent flags
$$\Phi = \Phi_0, \ \Phi_1, \dots, \ \Phi_r = \Lambda$$
such that $\Phi \cap \Lambda \subset \Phi_i$.
Let $J = \{j\hspace{.1cm}|\hspace{.1cm} (T_j,v) = (U_j,w)\}$ be the set of indices of the $j$-faces shared by $\Phi$ and $\Lambda$, and let $m$ denote the lowest non-negative index in $J$. (Such an $m$ exists since the connectivity of $\eG$ ensures that $n-1$ is an upper bound.) 
By the definition of $\PeGi$, there is a path 
$$v = v_0, \ v_1, \dots, \ v_p = w$$ 
from $v$ to $w$ in $\eG$ which preserves the tubing $T_m=U_m$. To prove that each poset is strongly flag-connected, we induct on the length $p$ of the path.

If $p=0$, then $v=w$ and we can appeal to the strong flag-connectivity of vertex-figures. If $p \geq1$, then $m \geq 1$ and $T_m$ is not a maximal tubing. By induction, we need to construct a sequence of flags from $\Phi = (\mathcal{T}, v)$ to (some yet to be specified) flag $\Lambda' = (\mathcal{T}', v_1)$. 
\begin{enumerate}
\item {\bf Selecting Flag $\Lambda'''$:} Since $v$ and $v_1$ are adjacent, their associated maximal tubings $T_0$ and $S_0$ differ at precisely one tube. Let $S_1$ denote the shared $(n-2)$-tubing.  Since nodes $v$ and $v_1$ are contained in $(T_m,v) = (U_m,w)$, $T_m\subset T_0$ and $T_m\subset S_0$. Thus, $T_m\subseteq T_0 \cap S_0=S_1$ and $(S_1,v)\preceq (T_m,v)$.  This guarantees the existence of a flag $\Lambda'''$ that contains $\Phi \cap \Lambda$ and has $(S_1, v)$ as its 1-face.  Moreover, $\Phi$ and $\Lambda'''$ share a 0-face $(T_0,v)$. By the strong flag-connectivity of vertex-figures, there exists a sequence of flags from $\Phi$ to $\Lambda'''$ where each flag in the sequence contains $(\Phi \cap \Lambda) \subseteq (\Phi \cap \Lambda''')$. 

\item {\bf Selecting Flag $\Lambda''$:} Since $v_1$ is in the node set of $(S_1,v_1) = (S_1, v)$, there exists a flag $\Lambda''$ which is the same as $\Lambda'''$ except for its 0-face $(S_0,v_1)$. Accordingly, $\Lambda''$ is adjacent to $\Lambda'''$ and $\Phi \cap \Lambda \subset \Lambda''$.

\item {\bf Selecting Flag $\Lambda'$:} Define flag $\Lambda' = \{(S_0,v_1), \cdots (S_{n-1}, v_{n-1})\}$, where $S_0$ and $S_1$ are defined as above, and $S_j = T_j = U_j$ for all $j \in J$. Such a flag is guaranteed to exist because $U_j = T_j \subseteq T_m \subseteq S_1$ for all $j\in J$. By construction, $\Phi \cap \Lambda$ is also contained in $\Lambda'$. Again, since $\Lambda''$ and $\Lambda'$ share a 0-face, the strong flag-connectivity of vertex-figures dictates that there exists a sequence of flags from $\Lambda''$ to $\Lambda'$ such that each flag contains $(\Phi \cap \Lambda)\subseteq(\Lambda'' \cap \Lambda')$.
\end{enumerate}
Combining the flag sequences from $\Phi$ to $\Lambda'''$ to $\Lambda''$ to $\Lambda'$ with one from $\Lambda'$ to $\Lambda$ (guaranteed by the inductive hypothesis) produces the desired flag sequence.
\end{proof}

\subsection{Finishing Touches}

With this machinery, the main results are proven.

\begin{proof}[Proof of Theorem~\ref{t:main}]
Let $G$ be a simple graph with $n$ nodes and connectivity $k$, along with a color palette.
Each distinct ordering of $k-1$ colors from the palette determines an assignment of colors to the $k-1$ outer blocks of the universal tube of $G$. As such, each ordering corresponds to a unique poset in the collection $\PeG$.
Let $\PeGi$ be a poset of $\PeG$. By construction, conditions (1) and (2) are satisfied. Proposition~\ref{p:strongflagconn} guarantees strong flag-connectivity, satisfying condition (3) and Proposition~\ref{p:diamond} verifies the diamond condition, satisfying condition (4). 

The 1-skeleton of $\PeGi$ is isomorphic to its associated connected component of $\eG$. Moreover, since every node in the component of $\eG$ corresponds to a maximal tubing of $G$ and is adjacent to precisely $n-1$ maximal tubings, we can conclude that $\PeGi$ is simple. This proves that $\PeG$, and thus $\rKG$, is a collection of simple abstract polytopes of rank $n-1$.
\end{proof}

Finally, we prove that $G[n]$ is an abstract polytope.

\begin{proof} [Proof of Proposition~\ref{p:null}]
By employing the same strategy as used to prove Theorem \ref{t:main}, it suffices to prove that the exchange graph of $G[n]$ is connected, which in turn requires showing that any two colors within a given maximal tubing can be exchanged. Let $T$ be a maximal tubing of $G[n]$ and let $v_n$ be the node without a tube. To exchange the colors of the tubes on nodes $v_i$ and $v_j$, remove the tube around $v_i$ and place that tube around $v_n$. Move the tube around $v_j$ to the now open node $v_i$ and move the tube on $v_n$ to $v_j$, completing the color exchange.
\end{proof}

%
%
\bibliographystyle{amsplain}

\end{document}